\documentclass[12pt]{article}
\usepackage{amsfonts,epsf,amsmath,amssymb,graphicx,amsthm,fullpage}

\usepackage{graphicx}
\usepackage{subfigure}
\usepackage{caption}

\newtheorem{theorem}{Theorem}[section]
\newtheorem{corollary}[theorem]{Corollary}
\newtheorem{lemma}[theorem]{Lemma}

\newtheorem{conjecture}[theorem]{Conjecture}
\newtheorem{problem}[theorem]{Problem}
\theoremstyle{definition}
\newtheorem{definition}[theorem]{Definition}

\newtheorem{question}[theorem]{Question}
\newtheorem{observation}[theorem]{Observation}

\def\cF{{\mathcal F}}
\def\cS{{\mathcal S}}
\def\NN{{\mathbb N}}
\def\esub{\subseteq}
\def\VEC#1#2#3{#1_{#2},\ldots,#1_{#3}}

\def\Hb{\overline{H}}

\def\Tb{\overline{T}}
\def\cD{\mathcal{D}}
\def\FL#1{\left\lfloor #1\right\rfloor}
\def\CL#1{\left\lceil #1\right\rceil}
\def\FR#1#2{\frac{#1}{#2}}
\def\CH#1#2{\binom{#1}{#2}}
\def\SM#1#2{\sum_{{#1}\in{#2}}}
\def\SE#1#2#3{\sum_{{#1}={#2}}^{#3}}
\def\st{\colon\,}
\def\nul{\varnothing}

\begin{document}

\title{Reconstruction from $k$-decks for graphs with maximum degree 2}
\author{
Hannah Spinoza\thanks{
Department of Mathematics, University of Illinois, Urbana IL 61801, U.S.A;
kolbhr@gmail.com.}\, and
Douglas B. West\thanks{
Departments of Mathematics, Zhejiang Normal University, Jinhua 321004, China,
and University of Illinois, Urbana IL 61801, U.S.A.; west@math.uiuc.edu.
Research supported by Recruitment Program of Foreign Experts, 1000 Talent Plan,
State Administration of Foreign Experts Affairs, China.} 
}

\maketitle

\vspace{-2pc}

\begin{abstract}
The {\it $k$-deck} of a graph is its multiset of induced subgraphs on
$k$ vertices.  We prove that $n$-vertex graphs with maximum degree $2$ have the
same $k$-decks if each cycle has at least $k+1$ vertices, each path component
has at least $k-1$ vertices, and the number of edges is the same.  Using this
for lower bounds, we obtain for each graph with maximum degree at most $2$ the
least $k$ such that it is determined by its $k$-deck.  For the $n$-vertex cycle
this value is $\FL{n/2}$, and for the $n$-vertex path it is $\FL{n/2}+1$.
Also, the least $k$ such that the $k$-deck of an $n$-vertex graph always
determines whether it is connected is at least $\FL{n/2}+1$.

\end{abstract}

\baselineskip16pt

\section{Introduction} \label{sec:intro}
The famous Reconstruction Conjecture of Kelly~\cite{kel1,kel2} and
Ulam~\cite{U} has been open for more than 50 years.  A {\it card} of a graph
$G$ is a subgraph of $G$ obtained by deleting one vertex.  Cards are unlabeled,
so only the isomorphism class of a card is given.  The {\it deck} of $G$ is the
multiset of all cards of $G$.  A graph is {\it reconstructible} if it is
uniquely determined by its deck. 

\begin{conjecture}[The Reconstruction Conjecture; Kelly~\cite{kel1,kel2},
Ulam~\cite{U}]
Every graph having more than two vertices is reconstructible. 
\end{conjecture}

We require more than two vertices since both graphs on two vertices have the
same deck.  Graphs in many families are known to be reconstructible; these
include disconnected graphs, trees, regular graphs, and perfect graphs.
Surveys on graph reconstruction
include~\cite{Bondy91,BondyHemminger77,Lauri97,Maccari02}.

Various parameters have been introduced to measure the difficulty of 
reconstructing a graph.  Harary and Plantholt~\cite{HP} defined the
{\it reconstruction number} of a graph to be the minimum number of cards from
its deck that suffice to determine it, meaning that no other graph has the same
multiset of cards in its deck.  All trees with at least five vertices have
reconstruction number $3$ (Myrvold~\cite{Myrvold90}), and almost all graphs have
reconstruction number $3$ (Bollob\'as~\cite{Bol}).  Since $K_{r,r}$ and
$K_{r+1,r-1}$ have $r+1$ common cards, the reconstruction number of an
$n$-vertex graph can be at least as large as $\FR n2+2$
(Myrvold~\cite{Myrvold89}).

Kelly looked in another direction, considering cards obtained by deleting
more vertices.  He conjectured a more detailed version of the Graph
Reconstruction Conjecture.

\begin{conjecture}[Kelly~\cite{kel2}]
For $\ell\in\NN$, there is an integer $f(\ell)$ such that any graph with at
least $f(\ell)$ vertices is reconstructible from its deck of cards obtained by
deleting $\ell$ vertices.
\end{conjecture}

\noindent
The Graph Reconstruction Conjecture is the claim $f(1)=3$ in this conjecture.

A {\it $k$-card} of a graph is an induced subgraph having $k$ vertices.
The \emph{$k$-deck} of $G$, denoted $\cD_k(G)$, is the multiset
of all $k$-cards.  Since each induced subgraph with $k-1$ vertices arises
exactly $n-k+1$ times by deleting one vertex from a member of $\cD_k(G)$,
we have the following.

\begin{observation}\label{k-1}
For any graph $G$, the $k$-deck $\cD_k(G)$ determines the $(k-1)$-deck
$\cD_{k-1}(G)$.
\end{observation}

Thus decks of larger cards provide at least as much information as decks of
smaller cards.  Graphs are ``easier'' to reconstruct if they can be
reconstructed from smaller cards.

\begin{definition}
A graph $G$ is {\it $k$-deck reconstructible} if no other graph has the same
$k$-deck.  Let $\rho(G)$ denote the least $k$ such that $G$ is $k$-deck
reconstructible.
\end{definition}

In light of Observation~\ref{k-1}, it is useful to know what information about
a graph can be reconstructed from the $k$-deck for small fixed $k$.  Such 
information is also available when considering larger $k$.  For example, only
the numbers of edges and vertices are reconstructible from the $2$-deck.  At
the other end, Manvel~\cite{Manvel} proved that if $|V(G)|=n\ge6$, then one can
determine from the $(n-2)$-deck whether or not $G$ is connected, acyclic,
regular, or bipartite.  This has recently been improved in~\cite{SW}, where
the authors showed that connectedness can always be
determined from the $(n-3)$-deck.

For a graph $G$, the maximum degree $\Delta(G)$ is reconstructible from the
$(\Delta(G)+2)$-deck, since some $(\Delta(G)+2)$-card has a vertex of degree
$\Delta(G)$, but no $(\Delta(G)+2)$-card has a vertex of degree $\Delta(G)+1$.
This was strengthened by Manvel:

\begin{theorem}[Manvel~\cite{Manvel}]\label{thm: degreeSequence}\label{manvel}
The degree list of a graph $G$ with maximum degree $\Delta(G)$ is
reconstructible from $\cD_{\Delta(G)+2}$.
\end{theorem}

Manvel~\cite{Manvel} also showed that the result is sharp in a strong sense;
the maximum degree is not always determined by $\cD_{\Delta(G)+1}(G)$.  Let
$G_k$ be the forest $\SE i0{\FL{k/2}} \CH k{2i}K_{1,k-2i}$ (that is,
$\CH k{2i}$ stars with $k-2i$ edges for $0\le i\le \FL{k/2}$).  Also, let
$H_k=\SE i0{\FL{(k-1)/2}} \CH k{2i+1} K_{1,k-2i-1}$.  Note that $\Delta(G_k)=k$
and $\Delta(H_k)=k-1$.  Nevertheless, the two graphs have the same $k$-deck,
and hence $\Delta(H)$ cannot always be determined from $\cD_{\Delta(H)+1}(H)$.

With Theorem~\ref{manvel}, we already recognize from the $k$-deck whether a
graph has maximum degree $2$ (when $k\ge4$).  However, we will show that much
larger cards are needed to guarantee determining whether a graph with maximum
degree $2$ is connected.  In Problem 11898 of the American Mathematical
Monthly, Richard Stanley posed a question that begins to suggest the difficulty
of reconstructing $2$-regular graphs from their $k$-decks.

\begin{problem}[Stanley \cite{Stanley}]\label{stanley}
Let $n$ and $k$ be integers, with $n\ge k \ge 2$. Let $G$ be a graph with $n$
vertices whose components are cycles of length greater than $k$. Let $f_k(G)$
be the number of $k$-element independent sets of vertices of $G$. Show that
$f_k(G)$ depends only on $k$ and $n$. 
\looseness -1
\end{problem}

Let $s(G,H)$ denote the number of induced subgraphs of $G$ isomorphic to $H$.
Graphs $G$ and $G'$ have the same $k$-deck if and only if $s(G,H)=s(G',H)$ for
all $H$ with $k$ vertices.  In the language of reconstruction, Stanley's
problem asserts $s(G,\overline{K}_k)=s(G',\overline{K}_k)$ for $n$-vertex
$2$-regular graphs $G$ and $G'$ whose components have length greater than $k$,
where $K_k$ is the complete graph with $k$ vertices and $\Hb$ denotes the
complement of $H$.  Stanley's proposed solution of Problem~\ref{stanley} used
generating functions.  Our proof and generalization are bijective and relate to
reconstruction.

Problem~\ref{stanley} considers only subgraphs having no edges.  We will prove
the same conclusion for all subgraphs with $k$ vertices.  That is, $n$-vertex
$2$-regular graphs whose components have more than $k$ vertices all have the
same $k$-deck.  Our technique of proof further generalizes to graphs with
maximum degree $2$.

\begin{theorem}\label{main} 
Let $G$ and $G'$ be graphs with maximum degree $2$ having the same number
of vertices and the same number of edges.  If every component in each graph is
a cycle with more than $k$ vertices or a path with at least $k-1$ vertices,
then $\cD_k(G)=\cD_k(G')$.
\end{theorem}

The essence of the theorem, and in fact the way we prove it, is what it says
for graphs with one or two components.  Let $G+H$ denote the disjoint union of
graphs $G$ and $H$, and let $C_n$ and $P_n$ denote the $n$-vertex cycle
and path.  The theorem includes

\smallskip
(1) $\cD_k(C_{q+r})=\cD_k(C_q+C_r)$ if $q,r\ge k+1$,

(2) $\cD_k(P_{q+r})=\cD_k(C_q+P_r)$ if $q\ge k+1$ and $r\ge k-1$, and

(3) $\cD_k(P_{q-1}+P_r)=\cD_k(P_q+P_{r-1})$ if $q,r\ge k$.
\smallskip

\noindent
These statements yield the following result.

\begin{corollary}\label{maincor}
For $n\ge3$, the least $k$ such that connectedness of an $n$-vertex graph $G$
can always be determined from its $k$-deck is at least $\FL{n/2}+1$ (even when
given $\Delta(G)=2$).  Furthermore, $\rho(P_n)=\FL{n/2}+1$ and
$\rho(C_n)=\FL{n/2}$ when $n\ge6$.
\end{corollary}
\begin{proof}
By (2), $\cD_k(P_n)=\cD_k(C_{\CL{n/2}+1}+P_{\FL{n/2}-1})$ when $k\le\FL{n/2}$.
This proves the claim about connectedness and also $\rho(P_n)\ge\FL{n/2}+1$.

Consider $\cD_k(P_n)$ with $k={\FL{n/2}+1}$.  If $n\ge6$, then $k\ge4$, and by
Theorem~\ref{manvel} we can reconstruct the degree list.  The components of
any reconstruction $G$ are cycles and one path.  Since the $k$-deck has no
cycle, $G$ can only have one cycle, and its length must exceed $k$.  Now the
path component has fewer than $\FL{n/2}-1$ vertices.  In $\cD_k(P_n)$, there
are $n-k+1$ copies of $P_k$.  However, when $l<k-1$, in $\cD_k(P_l+C_{n-l})$
there are $n-l$ copies of $P_k$, which is larger than in $P_n$.  Hence the deck
differs from $\cD_k(P_n)$ unless $G=P_n$.

By (1), $\cD_k(C_n)=\cD_k(C_{\CL{n/2}}+C_{\FL{n/2}})$ when $k<\FL{n/2}$.
Suppose $k=\FL{n/2}$.  If $n\ge8$, then $\FL{n/2}\ge4$ and by
Theorem~\ref{manvel} we can reconstruct the degree list.  Any $2$-regular graph
other than $C_n$ has a cycle of length at most $\FL{n/2}$, and this can be seen
in the $\FL{n/2}$-deck.

For $n\in\{6,7\}$, reconstruction of $C_n$ from the $3$-deck requires a
different argument.  We know the number of edges of any reconstruction $G$
from the $2$-deck, and we know the number of incidences (corresponding to edges
in the line graph) from the $3$-deck.  This yields
$\SM v{V(G)} \CH{d(v)}2 = n=\SM v{V(G)}\FR{d(v)}2$.  Now it is a standard
exercise by convexity that $G$ is $2$-regular.  Again a cycle will appear in
the $3$-deck if $G\ne C_n$.
\end{proof}

When $n=5$, the graphs $P_5$ and $C_4+P_1$ have the same $3$-deck, so the
condition $n\ge6$ in Corollary~\ref{maincor} cannot be weakened.  There are
also three pairs of $7$-vertex graphs that have the same $4$-deck, but all six
graphs are connected.  Possibly the threshold $k\ge\FL{n/2}+1$ for guaranteed
recognizability of connectedness is sharp when $n\ge6$.

\begin{question}
For $n\in\NN$, what is the least $k$ such that for every $n$-vertex graph $G$,
it can be determined from $\cD_k(G)$ whether $G$ is connected?  In particular,
does $\FL{n/2}+1$ suffice when $n\ge6$?  Does $n-4$ suffice?
\end{question}

N\'ydl~\cite{Nydl} proved that for any $n_0\in\NN$ and $0<q<1$, there exist
nonisomorphic graphs of some order $n$ larger than $n_0$ that have the
same $\FL{qn}$-deck.  However, connectedness is much less information to
request than the isomorphism class, and it remains possible that
$\FL{n/2}+1$ is a threshold for $k$ such that $\cD_k(G)$ always determines
whether $G$ is connected.

Sections 2 and 3 are devoted to the proof of Theorem~\ref{main}, which via
facts (1,2,3) yield lower bounds on $\rho(G)$ whenever $\Delta(G)=2$.
In Section 4 we prove that these lower bounds are optimal, giving procedures
to reconstruct $G$ from its $\rho(G)$-deck in all cases.  Here we give only
a simplified statement of the result.  The parameter $\epsilon'$ in this
statement depends on which paths are components of $G$, as detailed in
Theorem~\ref{rho}.

\begin{theorem}
If $\Delta(G)=2$, then $\rho(G)=\max\{\FL{m/2}+\epsilon,m'+\epsilon'\}$,
where $m$ is the number of vertices in a largest component $H$ of $G$,
$m'$ is the number of vertices in a largest component of $G-V(H)$ (possibly
$m'=0$), $\epsilon$ is $1$ if $G$ has $P_m$ as a component and otherwise
$0$, and $\epsilon'\in\{0,1,2\}$.
\end{theorem}

In particular, if $G$ is $2$-regular, then the full statement yields
$\epsilon'=0$ and $\rho(G)=\max\{\FL{m/2},m'\}$.

\section{Common $k$-Decks for Linear Forests}

A useful technical lemma implies that when two graph have the same $k$-deck,
taking the disjoint union of either with a third graph again yields two graphs
with the same $k$-deck.  This will allow us to change one or two components
of a graph while keeping the rest of the graph unchanged.  Note that $G[X]$
denotes the subgraph of $G$ induced by a vertex subset $X$.

\begin{lemma}\label{lem: two components}\label{samedeck}
If $G$, $G'$, and $H$ are graphs, then $\cD_k(G)=\cD_k(G')$ if and only if
$\cD_k(G+H)=\cD_k(G'+H)$.
\end{lemma}

\begin{proof}
Given a graph $F$, let $\cS_k(F)$ denote the set of labeled induced subgraphs
with at most $k$ vertices.  If $\cD_k(G)=\cD_k(G')$, then there is a bijection
$g$ from $\cS_k(G)$ to $\cS_k(G')$ that pairs isomorphic subgraphs.
It suffices to find such a bijection $h$ from $\cS_k(G+H)$ to $\cS_k(G'+H)$.
Given a set $U\esub V(G)\cup V(H)$, let $X=U\cap V(G)$ and $Y=U\cap V(H)$.
Note that $|X|,|Y|\le k$.  Hence we may define $h(U) =g(G[X])+H[Y]$.  In
fact, $h$ is a bijection, and $G[X]+H[Y] \cong g(G[X])+H[Y]$, so
$\cD_k(G+H)=\cD_k(G'+H)$.

Conversely, suppose that $\cD_k(G+H)=\cD_k(G'+H)$.  By Observation 1.3, we also
have $\cD_j(G+H)=\cD_j(G'+H)$ for $j\le k$.  Let $X$ be a graph with $k$
vertices and $r$ components.  We claim $s(G,X)=s(G',X)$, by induction on $k+r$.
If $r=1$, then ${s(G,X)}={s(G+H,X)-s(H,X)}={s(G'+H,X)-s(H,X)}={s(G',X)}$.
Let $[r]=\{1,\ldots,r\}$.  For $r>1$, let $\VEC X1r$ be the components
of $X$.  For $T\esub [r]$, let $X_T$ denote the disjoint union of
$\{X_i\st i\in T\}$, and let $\Tb=[r]-T$.  Using the induction hypothesis, we
compute
\begin{align*}
s(G,X)&=s(G+H,X)-\sum_{\nul\ne T\esub[r]}s(H,X_T)s(G,X_{\Tb})\\
&=s(G'+H,X)-\sum_{\nul\ne T\esub[r]}s(H,X_T)s(G',X_{\Tb})
~=~s(G',X).
\end{align*}
Thus $\cD_k(G)=\cD_k(G')$.
\end{proof}

We will use this lemma in both directions.  In one direction, it tells us
that any lower bound on $\rho(G)$ is also a lower bound on $\rho(G+H)$.
In the other, it tells us that when two graphs with the same $k$-deck have a
common component, deleting the shared component leaves two smaller graphs with
the same $k$-deck.

When we consider only graphs where every cycle has length larger than $k$,
every $k$-card is a linear forest, meaning a disjoint union of paths.  It will
be simpler to prove the equal-deck result first for linear forests.
To discuss linear forests precisely, we introduce helpful notation.

\begin{definition}
Let $L$ denote a list $\VEC\ell1p$ of distinct positive integers, let $m$
denote $\VEC m1p$, and let $L^m$ denote the linear forest having $m_i$
components isomorphic to $P_{\ell_i}$, for $1\le i\le p$.  Let $L^m_i$ denote
the linear forest obtained from $L^m$ by deleting a component isomorphic to
$P_{\ell_i}$, and let $L^m_{i,j}$ denote the result of deleting components
isomorphic to $P_{\ell_i}$ and $P_{\ell_j}$ (we allow $i=j$ when $m_i\ge2$).
Again $s(G,H)$ is the number of induced subgraphs of $G$ isomorphic to $H$,
and let $s'(G,H)$ be the number of induced subgraphs of $G$ isomorphic to $H$
in which a specified vertex of $G$ is used as an isolated vertex in $H$.
\end{definition}

We consider $s'(G,H)$ only when $H$ has an isolated vertex.  The vertex
specified in $G$ does not appear in the notation $s'(G,H)$, because we will
prove next that under appropriate conditions the value is the same for a range
of vertices.  For the remainder of this section, let the vertices of $P_n$ be
$\VEC w1n$ in order.

\begin{lemma}\label{lem: paths k recon}
Let $L^m$ be a linear forest with $k$ vertices.  For each specified vertex
$w_h$ such that $k\le h\le n-k+1$, the quantity $s'(P_n,L^m)$ has the same
value.
\end{lemma}
\begin{proof}
We use induction on $k$.  When $k=1$, there is exactly one copy of $P_1$
containing any specified vertex.  For $k>1$, the value is $0$ unless $L^m$ has
an isolated vertex.

We compare $s'(P_n,L^m)$ with $s'(C_n,L^m)$, where $C_n$ is obtained by adding
the edge $w_nw_1$.  By symmetry, $s'(C_n,L^m)$ is independent of the specified
vertex.  Note that $s'(C_n,L^m)$ does not count copies of $L^m$ in $P_n$ in
which some path starts with $w_1$ and another ends with $w_n$.  On the other
hand, it does count unwanted subgraphs that use the edge $w_nw_1$.

Note that $w_h$ is far enough from the ends of $P_n$ that there is room for
$P_{\ell_i}$ containing $w_1$ and $P_{\ell_j}$ containing $w_n$ without
touching $w_h$.  Also, in $C_n$ the edge $w_nw_1$ may occupy any of $\ell_i-1$
positions within a copy of $P_{\ell_i}$.  Summing over all the possible orders
of the paths or path using $w_1$ and $w_n$, we thus obtain the following
relation.
$$
s'(P_n,L^m) = s'(C_n,L^m)
+\sum_{i,j} s'(P_{n-(\ell_i+\ell_j+2)},L^m_{i,j})
-\sum_i (\ell_i -1)s'(P_{n-(\ell_i+2) }, L^m_i)
$$
Two extra vertices are deleted in each term to separate components of $L^m$.
In the middle sum, $i=j$ is allowed when $m_i\ge2$, and the set $\{i,j\}$
yields two terms when $i\ne j$; this sum is empty when $L^m$ consists of
only one path.  The final sum is actually a double-sum; we will show that the
summand in the inner sum with $\ell_i-1$ terms is constant.

By symmetry, $s'(C_n,L^m)$ is independent of $h$.  To obtain the same
conclusion for the other terms, we check the conditions in the statement of
the induction hypothesis.

For terms in the double sum, deleting $P_{\ell_i}$ and the neighboring vertex
from the beginning of $P_n$ leaves the vertex $w_h$ with a new index $h'$ in
$P_{n-\ell_i-\ell_j-2}$.  With $P_{\ell_i}$ containing $w_1$, we obtain
$h'=h-\ell_i-1$.  We have $h-\ell_i-1\ge k-(\ell_i+\ell_j)$ since $h\ge k$ and
$\ell_j\ge1$.  Similarly,
$$
h-\ell_i-1 \le n-(\ell_i+\ell_j+2) - \left( k- (\ell_i+\ell_j)\right)+1,
$$
since $h\le n-k+1$ and $\ell_i\ge1$.

The last sum is actually also a double sum, but the induction hypothesis
guarantees that the terms in the inner sum are equal.  When considering the
terms involving $\ell_i$, we lose at most $(\ell_i-1)+1$ vertices at the
beginning of the path, yielding $h'\ge h-\ell_i\ge k-\ell_i$.  Similarly,
we lose at most $\ell_i$ vertices from the end of the path and the index must
decrease at least by $2$, so $h'\le h-2\le (n-\ell_i-2)-( k-\ell_i)+1$.

By the induction hypothesis, all contributions are independent of the choice
of the specified vertex when it is in the given range.
\end{proof}

Note that we never need the value of $s'(P_n,L^m)$.
Lemma~\ref{lem: paths k recon} enables us to prove the special case of
Theorem~\ref{main} for linear forests.

\begin{theorem}\label{thm: paths k recon}
Let $L^m$ be a linear forest with $k$ vertices.  For an $n$-vertex graph $G$
that is a disjoint union of paths, each with at least $k-1$ vertices, the
number of induced copies of $L^m$ depends only on $L^m$, $n$, and $|E(G)|$. 
\end{theorem}

\begin{proof}
Given $n$, fixing $|E(G)|$ is equivalent to fixing the number of components.
By keeping all but two components fixed and applying
Lemma~\ref{lem: two components}, it therefore suffices to show
${s(P_{q-1}+P_r,L^m)=s(P_q+P_{r-1},L^m)}$ for $q,r\ge k$. 

Consider $P_{q+r+2}$ with $V(P_{q+r+2}) = \{ w_1,...w_{q+r+2} \}$.  Deleting
$\{w_q,w_{q+1},w_{q+2}\}$ yields $P_{q-1}+P_r$, while deleting
$\{w_{q+1},w_{q+2},w_{q+3}\}$ yields $P_q+P_{r-1}$.  Thus
$s(P_{q-1}+P_r,L^m) = s'(P_{q+r+2},L^m+P_1)$ when specifying $w_{q+1}$, while
$s(P_q+P_{r-1},L^m)=s'(P_{q+r+2},L^m+P_1)$ when specifying $w_{q+1}$.  By
Lemma~\ref{lem: paths k recon}, we need only capture $q+1$ and $q+2$ in
the given range.  

We have $|V(L^m+P_1))|=k+1$ and apply Lemma~\ref{lem: paths k recon} with
$n=q+r+2$.  Since $q,r\ge k$,
$$
|V(L^m+P_1))|=k+1\le q+1<q+2=n-r\le n-k= n-|V(L^m+P_1))|+1,
$$
as desired. 
\end{proof}

\begin{corollary}\label{cor: paths}
If $G$ and $G'$ are linear forests with the same number of vertices and same
number of edges whose components have at least $k-1$ vertices, then
$\cD_k(G)=\cD_k(G')$.
\end{corollary}

\section{Common $k$-Decks for Maximum Degree $2$}

We can extend the results to allow cycles because deleting any vertex of a
cycle leaves the same path.  Again the problem will reduce to working with just
two components.

\begin{lemma}\label{lem: cycle path}
Let $L^m$ be a linear forest with $k$ vertices.  If $q\ge k+1$ and $r\ge k-1$,
then $s(P_{q+r},L^m) = s(C_q+P_r,L^m)$
\end{lemma}
\begin{proof}
Let $\VEC u1{q+r}$ be the vertices of $V(P_{q+r})$ in order.  Consider an
induced copy of $L^m$.  Either $u_q$ is not used, or it appears in a path of
some length $\ell_i$.  In the latter case let $t$ be the number of vertices
starting with $u_q$ that lie in the copy of $P_{\ell_i}$; the hypotheses on $q$
and $r$ allow $t$ to run from $1$ to $\ell_i$.  These possibilities yield
$$
s(P_{q+r},L^m)= s(P_{q-1}+P_r, L^m)
+ \SE i1p \SE t1{\ell_i} s(P_{q-(\ell_i-t)-2}+P_{r-t},L^m_i).
$$

Now consider a vertex $x$ on $C_q$ in $C_q+P_r$.  By symmetry, the choice of
$x$ does not matter.  As above, in a copy of $L^m$ the vertex $x$ may be
omitted or appear in a copy of $P_{\ell_i}$ for some $i$.  By symmetry, the
position of $x$ in its copy of $P_{\ell_i}$ does not matter, since deleting
$V(P_{\ell_i})$ and two additional unused vertices always leaves
$P_{q-\ell_i-2}$.  Thus
$$
s(C_q+P_r,L^m)= s(P_{q-1}+P_r,L^m)
+ \SE i1p \ell_i s( P_{q-\ell_i -2}+P_r,L^m_i)
$$

It suffices to prove that the right sides of these two equations are equal.
The first term is identical.  It remains to show 
$$
s(P_{q-\ell_i-2}+P_r,L^m_i)
=s(P_{q-(\ell_i-t)-2}+P_{r-t},L^m_i)
$$
for $1\le i\le p$ and $1 \le t \le \ell_i$.  Adding vertices
$w_{h-1},w_h,w_{h+1}$ to connect the two given paths shows that each such value
is $s'(P_n,L^m_i+P_1)$ for the specified vertex $w_h$ along the host
path with vertices $\VEC w1n$, where $n=q+r+1-\ell_i$.
Theorem~\ref{thm: paths k recon} states that the value does not depend on $h$
as long as $k'\le h\le n-k'+1$, where $k'$ is the number of vertices in the
desired linear forest.

Here $k'=k-\ell_i+1$ and $n=q+r+1-\ell_i$, so we seek
$k-\ell_i+1\le h\le q+r-k+1$.  The lowest value taken by $h$ is $q-\ell_i$, and
the highest is $q$ (when $t=\ell_i$).  Since $q\ge k+1$ and $r\ge k-1$, the
desired inequalities hold (and we cannot weaken the hypotheses).
\end{proof}

Lemma~\ref{lem: cycle path} and Lemma~\ref{lem: paths k recon} yield the 
desired result for graphs that are not $2$-regular.

\begin{corollary}\label{cor: paths}
Let $G$ and $G'$ be non-regular graphs with maximum degree $2$ that have the
same number of vertices and same number of edges.  If all cycles in $G$ and
$G'$ have more than $k$ vertices and all path components have at least $k-1$
vertices, then $\cD_k(G)=\cD_k(G')$.
\end{corollary}
\begin{proof}
Since $G$ and $G'$ are not regular, each has at least one path component.
Using Lemma~\ref{lem: cycle path} to absorb cycles into paths, each has the
same $k$-deck as some linear forest with the same numbers of vertices and edges
as it and with at least $k-1$ vertices in each component.  By
Corollary~\ref{cor: paths}, the resulting linear forests $H$ and $H'$ have the
same $k$-deck.
\end{proof}

It remains only to consider $2$-regular graphs, which was our original 
motivation.  The results from the earlier cases simplify the proof here.

\begin{theorem}
Let $L^m$ be a linear forest with $k$ vertices. For $n$-vertex graphs whose
components are cycles with at least $k+1$ vertices, the number of induced
copies of $L^m$ depends only on $L^m$ and $n$. 
\end{theorem}
\begin{proof}
In particular, for each such graph, we show that the number of induced copies
of $L^m$ is the same as in $C_n$.  It suffices to show
$s(C_{q+r},L^m)=s(C_q+C_r, L^m)$ when $q,r \ge k+1$; we can then iteratively
reduce the number of components without changing the $k$-deck.

Choose $x\in V(C_{q+r})$ and $y\in V(C_r)$.  We expand the two needed
quantities by considering the usage of $x$ and $y$ in induced copies of $L^m$.
In each case, the specified vertex may be omitted, or it may occur in a copy
of some path $P_{\ell_i}$.  In the latter case, it may occur with any position
in $P_{\ell_i}$, but the resulting number of subgraphs is the same for each
position, since deleting any $\ell_i$-vertex path from a cycle leaves a path
of the same length.  We thus have the following two expansions.
\begin{align*}
s(C_{q+r},L^m)
&=s(P_{q+r-1},L^m)+\SE i1p \ell_i s(P_{q+r-\ell_i -2},L^m_i)\\
s(C_q+C_r,L^m)
&=s(C_q+P_{r-1},L^m)+\SE i1p \ell_i s(C_{q}+P_{r-\ell_i -2},L^m_i)
\end{align*}

It suffices to use Lemma~\ref{lem: cycle path} to show that corresponding terms
on the right are equal.  Equality of the first terms follows from $q\ge k+1$
and $r-1 \ge k-1$, which hold by assumption.  For the other case it suffices to
have $q\ge k-\ell_i +1$ and $r-\ell_i-2\ge k-\ell_i-1$.  The first inequality
holds since $q\ge k+1$.  The second simplifies to $r\ge k+1$, which holds by
assumption.
\end{proof}

\begin{corollary}\label{cor: lowerbound}
Any two $n$-vertex graphs whose components are cycles with at least $k+1$
vertices have identical $k$-decks.
\end{corollary}

With Corollaries~\ref{cor: paths} and \ref{cor: lowerbound},
we have now proved Theorem~\ref{main}, our main result.

\section{$\rho(G)$ for Graphs with Maximum Degree 2 }

We first reduce the problem of $k$-deck reconstruction to the problem
of finding all components with more than $k$ vertices.  This generalizes
classical reconstruction of disconnected graphs, and it applies to all graphs.

\begin{lemma}\label{kplus}
If all the components with more than $k$ vertices in a graph $G$ can be
determined from $\cD_k(G)$, then $G$ is $k$-deck reconstructible.
\end{lemma}
\begin{proof}
It suffices to show that all the components with exactly $k$ vertices can
be determined, since we have already observed that $\cD_k(G)$ determines
$\cD_{k-1}(G)$.  We then iterate to find all smaller components.

Let $\VEC H1r$ be the components of $G$ with more than $k$ vertices.  Let $F$
be a component with exactly $k$ vertices.  The number of components of $G$
isomorphic to $F$ is obtained by subtracting $\SE i1r s(H_i,F)$ from the number
of cards in $\cD_k(G)$ isomorphic to $F$.
\end{proof}

\begin{lemma}\label{k1count}
If $\Delta(G)=2$, then the number of components of $G$ that are paths with
at least $k-1$ vertices is $s(G,P_{k-1})-s(G,P_k)-ks(G,C_k)$.
\end{lemma}
\begin{proof}
Each path component with at least $k-1$ vertices contributes exactly $1$ to
$s(G,P_{k-1})-s(G,P_k)$.  Each $m$-cycle with $m>k$ contributes $m$ to both
$s(G,P_{k-1})$ and $s(G,P_k)$.  Each $k$-cycle contributes $k$ to both
$s(G,P_{k-1})$ and $ks(G,C_k)$.  No smaller component contributes.  Hence
each component is counted correctly.
\end{proof}

\begin{lemma}\label{k1recon}
If $\Delta(G)=2$, then the number of components of $G$ that are paths with
at least $k-1$ vertices is determined by $\cD_k(G)$.
\end{lemma}
\begin{proof}
Each subgraph of $G$ having $k$ vertices appearances exactly once as a card
in $\cD_k(G)$.  Hence counting the cards that are paths and cycles yields
$s(G,P_k)$ and $s(G,C_k)$.  Each induced subgraph of $G$ that is a copy of
$P_{k-1}$ occurs as an induced subgraph of a $k$-card exactly $n-k+1$ times,
where $n=|V(G)|$.  Thus $s(G,P_{k-1})=s(J,P_{k-1})/(n-k+1)$, where $J$ is
the disjoint union of all the $k$-cards of $G$.  Hence we can determine all
the terms in the computation in Lemma~\ref{k1count}.
\end{proof}

\begin{lemma}\label{nopath}
Let $G$ be a graph with maximum degree $2$.  If $G$ has no path components with
at least $k-1$ vertices, and $0<s(G,P_k)\le 2k+1$, then $G$ has exactly one
component with more than $k$ vertices, and it is a cycle with $s(G,P_k)$
vertices.
\end{lemma}
\begin{proof}
By hypothesis, no components are paths with more than $k$ vertices, so such
components are cycles, each contributing at least $k+1$ cards that are $P_k$.
With $s(G,P_k)\le 2k+1$, there is at most one such component.  With
$s(G,P_k)>0$, there is at least one.
\end{proof}

\begin{lemma}\label{onepath}
Let $G$ be a graph with maximum degree $2$.  If $G$ has exactly one path
component with at least $k-1$ vertices, and $0\le s(G,P_k)\le k$, then $G$ has
no cycle with more than $k$ vertices, and its one path component with at
least $k-1$ vertices has $s(G,P_k)+k-1$ vertices.
\end{lemma}
\begin{proof}
Since $s(G,P_k)\le k$, no component is a cycle with more than $k$ vertices.
Since $s(C_k,P_k)=0$, all copies of $P_k$ come from paths, of which by
hypothesis there is only one.  Now $s(P_m,P_k)=m-k+1$ for $m\ge k-1$ completes
the proof.
\end{proof}

In order to use the lemmas above to prove the upper bounds, we need to
determine from $\cD_k(G)$ that $G$ has maximum degree $2$.  When $k\ge4$,
this follows from Manvel's result, but we will need it also sometimes
when $k=3$.  The cases in the next lemma will suffice.

\begin{lemma}\label{deg2}
If $\Delta(G)=2$, then every reconstruction from $\cD_3(G)$ has maximum 
degree $2$ in the following cases: $G$ has no isolated vertices,
$G=P_4+aP_1$ with $a\ge0$, or $G=aP_3+bC_3+cP_2+dP_1$ with $\min\{b,d\}\le3$
and $a\le1$.
When $G$ has an isolated vertex, there are alternative reconstructions 
with maximum degree $3$ in the following cases: $G$ has a component with at
least five vertices, or a $4$-cycle, or three components forming $P_4+C_3+P_1$,
or eight components forming $4C_3+4P_1$.  Let $\cF$ denote the family of 
such graphs $G$.
\end{lemma}
\begin{proof}
We first exhibit the alternative reconstructions for $G\in\cF$.  Let $Y_r$ be
any tree with $r$ vertices and three leaves.  Note that $\Delta(Y_r)=3$.

For $m\ge4$, the graph $C_m+P_1$ has the same $3$-deck as $Y_{m+1}$.  The
$3$-deck has no triangles, $m$ copies of $P_3$, and $m(m-4)$ copies of
$P_2+P_1$, with the other cards being $3P_1$.

For $m\ge5$, the graph $P_m+P_1$ has the same $3$-deck as $Y_{m-1}+P_2$.  The
$3$-deck has no triangles, $m-2$ copies of $P_3$, and $(m-2)^2+1$ copies of
$P_2+P_1$; the other cards are $3P_1$.

In addition, ${\cD_3(P_4+C_3+P_1)}={\cD_3(K_{1,3}^++2P_2)}$, where $K_{1,3}^+$
is the ``paw'', obtained from $K_{1,3}$ by adding one edge (the $3$-deck has
one triangle, two copies of $P_3$, and $29$ copies of $P_2+P_1$).  Also,
${\cD_3(4C_3+4P_1)}={\cD_3(K_4+6P_2)}$ (the $3$-deck has four triangles, no
copies of $P_3$, and $156$ copies of $P_2+P_1$).

For the remaining cases, let $H$ be a reconstruction from $\cD_3(G)$.  We know
$|V(H)|$ from $\cD_1(G)$ (call it $n$) and $|E(H)|$ from $\cD_2(G)$.  Also
$\cD_3(G)$ tells us the number of incidences between edges, which equals
$\sum_{v\in V(H)}\CH{d_H(v)}2$.  If $G$ has no isolated vertices, then $G$ has
$n-t/2$ edges and $n-t$ incidences, where $t$ is the number of vertices of
degree $1$.  Among all lists $\VEC d1n$ of nonnegative integers summing to
$2n-t$, by convexity $\sum \CH{d_i}2$ is minimized (and equals $n-t/2$)
precisely when all entries are $1$ or $2$.  Hence in this case we know the
maximum degree (and degree list) of $H$.

When $G=P_4+aP_1$, every reconstruction $H$ from $\cD_3(G)$ has three edges.
Thus $H$ consists of $P_4$, $K_{1,3}$, $C_3$, $P_3+P_2$, or $3P_2$ plus
isolated vertices.  Among these, only $G$ has exactly two copies of $P_3$ in
its $3$-deck.

It remains to consider $G=aP_3+bC_3+cP_2+dP_1$ with $a\le1$.
If $a=1$, then $\cD_3(G)$ has exactly one copy of $P_3$.  Being connected,
it comes from one component of $H$, and the only connected graph with exactly
one copy of $P_3$ in its $3$-deck is $P_3$.  Hence $H$ has $P_3$ as one
component.  By Lemma~\ref{samedeck}, we therefore need only consider
$G=bC_3+cP_2+dP_1$.  Let $H$ be an alternative reconstruction from the $3$-deck
of a minimal such graph $G$.  By Lemma~\ref{samedeck}, each graph is a
component in at most one of $G$ and $H$.

Since $P_3$ is not a $3$-card, $H$ is a disjoint union of complete graphs.
When $b>0$, we have that $C_3$ is not a component of $H$.  Hence $b$ counts
the triangles in the components of $H$ with more than three vertices.  In $G$,
we have three edges per triangle.  In $H$ the components generating triangles
have fewer than three edges per triangle.  Hence $H$ has isolated edges, and
$G$ does not.  A copy of $K_m$ in $H$ with $m>3$ uses $\CH m2$ edges to
generate $\CH m3$ triangles, which in $G$ use $3\CH m3$ edges.  Hence $H$ has
$3\CH m3-\CH m2$ isolated edges for each such component.  Associated with each
such component in $H$, we thus have $m+6\CH m3-2\CH m2$ vertices in $H$ and
$3\CH m3$ edges in $G$.  This requires at least $3\CH m3-m(m-2)$ isolated
vertices in $G$.  If $\Delta(H)\ne2$, then $H$ has a component with $m\ge4$,
which requires that $G$ has at least four isolated vertices and at least four
components that are triangles.

Finally, if $G=cP_2+dP_1$, then we know $G$ is reconstructible from $\cD_3(G)$.
\end{proof}

These exceptions in Lemma~\ref{deg2} will yield exceptions to the general
formula we now define.

\begin{definition}\label{kG}
Given a graph $G$ with $n$ vertices and maximum degree at most $2$, let $m$ and
$m'$ be the numbers of vertices in two largest components of $G$, with
$m\ge m'$ (possibly $m'=0$).  Let $\epsilon=1$ if $G$ has $P_m$ as a component;
otherwise $\epsilon=0$.
Let $\epsilon'=2$ if $m'<m-1$ and $G$ has $P_{m'}$ as a component.
Let $\epsilon'=1$ if $m'=m-1$ and $G$ has $P_{m'}$ as a component,
if $m'<m$ and $G$ has $P_{m'-1}$ but not $P_{m'}$ as a component,
or if $m'=m$ and at least two components of $G$ equal $P_m$.
Otherwise, let $\epsilon'=0$.  Now define
\begin{equation*}
k_G=\max\{\FL{m/2}+\epsilon,m'+\epsilon'\}.\tag{*}
\end{equation*}
\end{definition}

Now we can determine $\rho(G)$.

\begin{theorem}\label{rho}
Let $G$ be a graph with $n$ vertices and maximum degree at most $2$,
using notation $m,m',\epsilon,\epsilon',k_G$ as in Definition~\ref{kG}.
Always $\rho(G)=k_G$, except that $\rho(G)=4$ when $k_G=3$ and $G\in\cF$.
\end{theorem}

\begin{proof}
{\it Lower bounds.}
We first use facts (1),(2),(3) listed after Theorem~\ref{main}.  When we
provide another graph having the same $k$-deck, we obtain $\rho(G)>k$.

Consider first a largest component, and let $k=\FL{m/2}+\epsilon-1$.

\medskip
\qquad
(1) yields $\cD_k(C_m)=\cD_k(C_{\CL{m/2}}+C_{\FL{m/2}})$ when $k<\FL{m/2}$, and

\qquad
(2) yields $\cD_k(P_m)=\cD_k(C_{\CL{m/2}+1}+P_{\FL{m/2}-1})$ when
$k\le \FL{m/2}$.
\medskip

\noindent
Combined with Lemma~\ref{lem: two components}, we obtain
$\rho(G)\ge\FL{m/2}+\epsilon$.

Now consider two large components, and let $k=m'+\epsilon'-1$.  
Suppose first that $G$ has $P_q$ as a component, where $q\in\{m',m'-1\}$.

\medskip
\qquad
(2) yields $\cD_k(C_m+P_{q})=\cD_k(P_{m+q})$ when $k< m$ and $k\le q+1$,

\qquad
(3) yields $\cD_k(P_m+P_{q})=\cD_k(P_{m-1}+P_{q+1})$ when $k\le\min\{m,q+1\}$.

\medskip
\noindent
Depending on whether the unique largest component of $G$ is a path,
these observations yield $\rho(G)\ge m'+\epsilon'$ in these cases:
$\epsilon'=2$ (using $q=m'$), and $\epsilon'=1$ (when $m'<m$ using
$q=m'-1$ or $m'=m-1$ using $q=m'$, and when $m'=m$ using $q=m'=m$).

When $\epsilon'=0$, every component with $m'$ vertices is a cycle, except
when $G$ contains $C_m+P_m$.  This we can also write as $P_m+C_{m'}$, since
then $m'=m$.  Let $k=m'-1$.  Now

\qquad
(1) yields $\cD_k(C_m+C_{m'})=\cD_k(C_{m+m'})$ when $k<m'\le m$,

\qquad
(2) yields $\cD_k(P_m+C_{m'})=\cD_k(P_{m'+m})$ when $k<m'$ (since also
$k\le m+1$).

\medskip
\noindent
Combined with Lemma~\ref{lem: two components}, we obtain
$\rho(G)\ge m'+\epsilon'$ in each case.

\medskip
Thus $\rho(G)\ge k_G$.  When $k_G=3$ and $G\in\cF$, the alternative 
reconstructions in Lemma~\ref{deg2} show that $\rho(G)\ge4$.

\bigskip
{\it Upper bounds.}
If $|E(G)|\le1$, then $k_G=2$, and indeed $G$ is determined by its $2$-deck.
If $|E(G)|\ge2$ and $\Delta(G)=1$, then $k_G=3$, and by Manvel's result
$\cD_3(G)$ determines the degree list, which in turn determines $G$.  In all
other cases, $\Delta(G)=2$ and $k_G\ge3$.
If $k_G=3$ and $G$ has an isolated vertex with $m\ge4$ (except $P_4+aP_1$)
or with $G$ containing $4C_3+4K_1$, then set $k=4$.  Otherwise, set 
set $k=k_G$.

When $k_G=3$ and $G\notin\cF$, every reconstruction from $\cD_3(G)$ has maximum
degree $2$, by Lemma~\ref{deg2}.  In all other cases, $k\ge4$ and
Manvel's result implies that every reconstruction has maximum degree $2$.
This fact is all we need for the main argument.

By Lemma~\ref{kplus}, it suffices to show that $\cD_k(G)$ determines the
components of $G$ with more than $k$ vertices.  Since $k\ge k_G$, we have
$k\ge\FL{m/2}+\epsilon$ and $k\ge m'+\epsilon'$.
The key claim that allows us to apply the lemmas is this:  

\medskip
{\narrower

\noindent
Claim: If $k\ge m'+\epsilon'$ and $m'<m-1$ (or $m'=m-1$ and $G$ does not have
$P_{m'}$ as a component), then at most one path component has at least $k-1$
vertices.

}

\medskip
\noindent
We check cases.  If $\epsilon'=2$, then $G$ has $P_{m'}$ as a component and at
most one component with more vertices, which suffices since $m'<k-1$.
If $\epsilon'=1$ and $G$ has $P_{m'-1}$ but not $P_{m'}$ as a component, then
at most one component that is a path has at least $m'$ vertices, which
suffices since $m'\le k-1$.  If $m'=m-1$ and $G$ does not
have $P_{m'}$ or $P_{m'-1}$ as a component, then $\epsilon'=0$ and 
$G$ has at most one path component with at least $k-1$ vertices.
The claim applies to all cases with $m'<m$ except when $m'=m-1$ and 
$G$ has $P_{m'}$ as a component.

For all these cases, $G$ has at most one path component having at least $k-1$
vertices.  Since $\epsilon'=2$ only when $m'<m-1$, whenever $m'<m$ we also
have $k\le m$.  Hence there is one such path component if $P_m$ is a component,
in which case $s(G,P_k)=m-k+1$, and there are none if $C_m$ is a component and
$P_m$ is not, in which case $s(G,P_k)=m$.

Now consider a reconstruction $H$ from $\cD_k(G)$.  By Lemma~\ref{k1recon}, the
number of components of $H$ that are paths with at least $k-1$ vertices is the
same as in $G$.  Furthermore, $s(H,P_k)=s(G,P_k)$; this just counts the
$k$-cards isomorphic to $P_k$.

When $G$ has no components that are paths with at least $k-1$ vertices,
cards that are paths arise only from cycles with more than $k$ vertices.
In particular, since $k\ge m'$, no such cards arise from $m'$-cycles, and
$m=s(G,P_k)=s(H,P_k)$.  Since
$k\ge \FL{m/2}+\epsilon$ and here $\epsilon=0$, we have $m\le 2k+1$.  Now
Lemma~\ref{nopath} implies that $H$ has exactly one
component with more than $k$ vertices, and it is $C_m$.

When $G$ has exactly one component that is a path with at least $k-1$ vertices,
and it is $P_m$, the same holds for $H$.  Again $k\ge m'$ implies that no
copies of $P_k$ areise from $m'$-cycles, so $m-k+1=s(G,P_k)=s(H,P_k)$.
$k\ge \FL{m/2}+\epsilon\le m'+\epsilon'$ and $\epsilon=1$, we have $m\le 2k-1$,
and hence $m-k+1\le k$.  Now Lemma~\ref{onepath} implies that $H$ has exactly
one component with more than $k$ vertices, and it is $P_m$.

In each case above the components of $H$ having more than $k$ vertices are the
same as in $G$, which suffices.  In the remaining cases we show that both have
no such components.  These cases are when $m'=m$ or when $m'=m-1$ with $P_{m'}$
being a component of $G$.

If $G$ has at least two components isomorphic to $P_m$, then $\epsilon'=1$ and
$k=m+1$.  Since no component of $G$ has at least $k$ vertices, no card is
connected; hence $H$ has no component with at least $k$ vertices.
Otherwise, $k=m$.  Since $G$ has no component with more than $k$ vertices, at
most one $k$-card is $P_k$.  Thus $s(H,P_k)\le 1$.  Since $\Delta(H)=2$, we
again conclude that $H$ has no component with more than $k$ vertices.
\end{proof}


\end{document}